\documentclass[12pt]{amsart}

\usepackage{latexsym}
\usepackage{amsfonts}
\usepackage{amssymb}
\usepackage{fullpage}
\usepackage[OT4]{fontenc}
\usepackage{amscd}
\usepackage{textcomp}
\usepackage{color}

\newtheoremstyle{s2}{9pt}{9pt}{\rm}{\parindent}{\bf}{.}{0.5em}{}
\theoremstyle{s2} 

\newtheorem{definition}{Definition}

\newtheoremstyle{s1}{9pt}{9pt}{\it}{\parindent}{\bf}{.}{0.5em}{}
\theoremstyle{s1}
\newtheorem{lemma}[definition]{Lemma}
\newtheorem{theorem}[definition]{Theorem}
\newtheorem{corollary}[definition]{Corollary}
\newtheorem{proposition}[definition]{Proposition}
\newtheorem{remark}[definition]{Remark}
\newtheorem{observation}[definition]{Observation}

\newtheorem*{conjecture*}{Conjecture}

\def\P2{{\mathbb{P}^2}}
\def\P1{{\mathbb{P}^1}}
\def\Ce{\mathbb{C}}

\def\Zet{\mathbb{Z}}
\def\En{\mathbb{N}}

\def\Oo{\mathcal{O}}

\DeclareMathOperator{\mult}{mult}
\DeclareMathOperator{\lcm}{lcm}
\DeclareMathOperator{\Num}{Num}

\numberwithin{definition}{section}

\title[On $k$-jet ampleness of line bundles]
{On $k$-jet ampleness of line bundles on hyperelliptic surfaces} \makeatletter

\@addtoreset{equation}{section}
\makeatother
\author{{\L}ucja Farnik}
\address{{\L}ucja Farnik, Jagiellonian University, Faculty of Mathematics and Computer Science, {\L}ojasiewicza~6, 30-348 Krak\'{o}w, Poland} 
\email{lucja.farnik@uj.edu.pl}

\date{\today}

\begin{document}
\bibliographystyle{alpha}

\begin{abstract}
We study $k$-jet ampleness of line bundles on hyperelliptic surfaces using vanishing theorems. Our main result states that on a hyperelliptic surface of an arbitrary type a line bundle of type $(m,m)$ with $m\geq k+2$ is $k$-jet ample.
\end{abstract}

\maketitle

\thispagestyle{empty}
\section{Introduction}

The concepts of higher order embeddings: $k$-spandness, $k$-very ampleness and $k$-jet ampleness were introduced and studied in a series of papers by M.C. Beltrametti, P. Francia and A.J. Sommese, see \cite{BeFS1989}, \cite{BeS1988}, \cite{BeS1993}. The last notion is of our main interest in the present work.

The problem of $k$-jet ampleness has been  studied on certain types of algebraic surfaces. In \cite{BaSz2-1997} Th. Bauer and T. Szemberg characterise $k$-jet ample line bundles on abelian surfaces with Picard number 1.
For an ample line bundle $L$ on a $K3$ surface Th.~Bauer, S.~Di Rocco and T. Szemberg in \cite{BaDRSz2000}, and S. Rams and T. Szemberg in \cite{RSz2004} explore for which $n$ the line bundle $nL$ is $k$-jet ample.

There are also several papers concerning $k$-jet ampleness in higher dimensions, e.g. \linebreak \cite{BaSz1997} study $k$-jet ampleness on abelian varieties, \cite{DR1999} on toric varieties, \cite{BeSz2000} on Calabi-Yau threefolds, \cite{ChI2014} on hyperelliptic varieties.

In the present paper we prove that on a hyperelliptic surface of an arbitrary type a line bundle of type $(m,m)$ with $m\geq k+2$ is $k$-jet ample. Note that a line bundle of type $(k+2,k+2)$ is numerically equivalent to $(k+2)L_1$ where $L_1=(1,1)$. By theory of hyperellipic surfaces we know that $L_1$ is ample,  so our result is consistent with results obtained on other algebraic surfaces with Kodaira dimension $0$. 
Our approach uses vanishing theorems  of the higher order cohomology groups --- Kawamata-Viehweg Theorem and Norimatsu Lemma. 

Proof of the fact that a line bundle of type $(k+2,k+2)$  is $k$-jet ample on any hyperelliptic surface~$S$ can be found also in \cite{ChI2014}.
 The authors use the fact that  $S$ is covered by an abelian surface divided by the group action, and the results of \cite{PP2004}. 
We provide a self-contained and more elementary proof of this fact.

\section{Preliminaries}
Let us set up the notation and basic definitions. We work over the field of complex numbers~$\Ce$. We consider only smooth reduced and irreducible projective varieties. By $D_1\equiv D_2$ we denote the numerical equivalence of divisors $D_1$ and $D_2$. By a curve we understand an irreducible subvariety of dimension 1. 
In the notation we follow \cite{PAG2004}.

\vskip 5pt
Let  $X$ be a smooth projective variety of dimension $n$. Let  $L$ be a line bundle on  $X$, and let $x\in X$.

\begin{definition}~

\begin{enumerate}
\item We say that  $L$ generates $k$-jets at $x$, if the restriction map $$H^0(X,L)\longrightarrow
 H^0(X,L\otimes\Oo_X/m_x^{k+1})$$  is surjective.
\item We say that  $L$ is $k$-jet ample, if for every points $x_1$, $\ldots$, $x_r$ the restriction map
$$H^0(X,L) \longrightarrow H^0\left(X,L\otimes  \Oo_{X}/(m_{x_1}^{k_1}\otimes\ldots\otimes m_{x_r}^{k_r})\right)$$
is surjective, where $\sum_{i=1}^r k_i=k+1$.
\end{enumerate}
\end{definition}
Note that $0$-jet ampleness is equivalent to being spanned by the global sections, and $1$-jet ampleness is equivalent to very ampleness.

The notion of $k$-jet ampleness  generalises the notion of very ampleness and $k$-very ampleness (see \cite{BeS1993}, Proposition 2.2). We recall the definition of $k$-very ampleness as we mention this notion it the proof of the main theorem:
\begin{definition}%k-bardzo szerokoϾ
We say that a line bundle $L$ is $k$-very ample if for every $0$-dimensional subscheme $Z\subset X$ of length $k+1$ the restriction map
 $$H^0(X,L)\longrightarrow H^0(X,L\otimes\Oo_Z)$$ is surjective.
\end{definition}

In the other words $k$-very ampleness means that the subschemes of length at most $k+1$ impose independent conditions on global sections of $L$.

We also recall the definition of the Seshadri constant:
\begin{definition} 
The Seshadri constant of $L$ at a given point $x\in X$ is the real number
$$\varepsilon(L,x)=\inf \left\{\frac{LC}{\mult_xC}: \  C\ni x\right\},$$
where the infimum is taken over all irreducible curves 
$C\subset X$ passing through $x$.
\end{definition}

If $\pi\colon\widetilde{X}\longrightarrow X$ is the blow-up of  $X$ at $x$, and $E$ is an exceptional divisor of the blow-up, then equivalently the Seshadri constant may be defined as (see e.g. \cite{PAG2004} vol. I, Proposition 5.1.5):
$$\varepsilon(L,x)=\sup \left\{\varepsilon: \ \pi^*L-\varepsilon E \text{ is nef} \right\}.$$

\vskip 5pt

We will use two vanishing theorems for the higher order cohomology groups --- Kawamata-Viehweg Vanishing Theorem and Norimatsu Lemma.
\begin{theorem}[Kawamata-Viehweg Vanishing Theorem; \cite{Laz1997}, Vanishing Theorem~5.2]
Let $D$ be an nef and big divisor on $X$. Then
$$H^i(X, K_X+D)=0 \quad \text{ for } i>0.$$
\end{theorem}

\begin{definition}[\cite{PAG2004}, vol. II, Definition 9.1.7] We say that $D=\sum D_i$ is a simple normal crossing  divisor (or an SNC divisor for short) if $D_i$ is smooth for each $i$, and  $D$ is defined in a neighbourhood of any point in local coordinates $(z_1,\ldots, z_n)$ as $z_1\cdot\ldots\cdot z_k=0$ for some $k\leqslant n$.
\end{definition}

\begin{theorem}[Norimatsu Lemma; \cite{PAG2004}, vol. I, Vanishing Theorem 4.3.5]
 Let $D$ be an ample divisor on $X$ and let $F$ be an SNC divisor on $X$. Then
$$H^i(X, K_X+D+F)=0 \quad \text{ for } i>0.$$
\end{theorem}

\section{Hyperelliptic surfaces}

First let us recall the definition of a hyperelliptic surface.
\begin{definition}
A hyperelliptic surface $S$ (sometimes called bielliptic) 
is a surface with Kodaira  dimension equal to $0$ and  irregularity  $q(S)=1$.
\end{definition}
 
Alternatively (see \cite{Bea1996}, Definition VI.19), a surface $S$ is hyperelliptic if  $S\cong (A\times B)/G$, where $A$ and $B$ are elliptic curves, and $G$ is an abelian group acting on A by translation and acting on B, such that  $A/G$ is an elliptic curve and $B/G\cong \mathbb{P}^1$. $G$ acts on $A\times B$ coordinatewise.
Hence we have the following situation:
$$
\begin{CD}
S\cong (A\times B)/G @>\Phi>> A/G @.\\
@V\Psi VV @.\\
B/G\cong \P1
\end{CD}
$$
where $\Phi$ and $\Psi$ are the natural projections.

Hyperelliptic surfaces were classified at the beginning of 20th century by G. Bagnera and M.~de Franchis in \cite{BF1907}, and independently by F. Enriques i F. Severi in \cite{ES1909-10}. They showed that there are seven non-isomorphic types 
of hyperelliptic surfaces, characterised by the action of $G$ on $B\cong\Ce/(\Zet\omega\oplus\Zet)$ (for details see e.g. \cite{Bea1996}, VI.20). For each hyperelliptic surface we have that the canonical divisor $K_S$ is numerically trivial.

In 1990 F.~Serrano in \cite{Se1990}, Theorem 1.4, characterised the group of classes of numerically equivalent divisors $\Num(S)$
for each of the  surface's type:
\begin{theorem}[Serrano] A basis of the group  $\Num(S)$
for each of the hyperelliptic surface's type and the multiplicities of the singular fibres in each case are the following:
$$
\begin{array}{c|l|l|l}
\text{Type of a hyperelliptic surface}&G&m_1,\ldots,m_s&\text{Basis of $\Num(S)$}\\
\hline
1&\Zet_2&2,2,2,2&A/2, B\\
2&\Zet_2\times\Zet_2&2,2,2,2&A/2, B/2\\
3&\Zet_4&2,4,4&A/4, B\\
4&\Zet_4\times\Zet_2&2,4,4&A/4, B/2\\
5&\Zet_3&3,3,3&A/3, B\\
6&\Zet_3\times\Zet_3&3,3,3&A/3, B/3\\
7&\Zet_6&2,3,6&A/6, B
\end{array}
$$
\end{theorem}
Let $\mu=\lcm\{m_1, \ldots, m_s\}$ and let $\gamma=|G|$. Given a hyperelliptic surface, its basis of  $\Num(S)$ consists of divisors $A/\mu$ and $\left(\mu/\gamma\right) B$. 

We say that $L$ is a line bundle of type $(a,b)$ on a hyperelliptic surface if $L\equiv a\cdot A/\mu+b\cdot(\mu/\gamma) B$.
In  $\Num(S)$  we have that $A^2=0$, $B^2=0$, $AB=\gamma$. 
Note that a divisor  $b\cdot\left(\mu/\gamma\right) B\equiv (0,b)$, $b\in\Zet$, is effective if and only if  $b\cdot\left(\mu/\gamma\right)\in \En$ (see \cite{Ap1998}, Proposition~5.2).

The following proposition holds:
\begin{proposition}[see \cite{Se1990}, Lemma 1.3]\label{Ser2}
Let $D$ be a divisor of type $(a,b)$ on a hyperelliptic surface $S$. Then
\begin{enumerate}
\item  $\chi(D)=ab$;
\item $D$  is ample if and only if $a>0$ and  $b>0$;
\item If $D$  is ample then $h^0(D)=\chi(D)=ab$.
\end{enumerate}
\end{proposition}

\vskip 5pt
Now we recall a bound for the self-intersection of a curve.
The adjunction formula, applied to the normalisation of a curve $C$, implies the following formula:
\begin{proposition}[Genus formula, \cite{GH1978}, Lemma, p. 505]\label{gen_form} Let $C$ be a curve on a surface $S$, passing through $x_1$, $\ldots$, $x_r$ with multiplicities respectively $m_1$, $\ldots$, $m_r$. Let $g(C)$ denote the genus of the normalisation of $C$.
Then
$$g(C)\leq \frac{C^2+C.K_{S}}{2}+1-\sum_{i=1}^r\frac{m_i(m_i-1)}{2}.$$
\end{proposition}

Note that:
\begin{observation}\label{gen_form_hyperell}
A curve $C$ on a hyperelliptic surface has  genus at least 
 $1$. Indeed, otherwise the normalisation of $C$, of genus zero, would be a covering (via  $\Phi$) of an elliptic curve  $A/G$. This contradicts the Riemann-Hurwitz formula.
\end{observation}

\section{Main result}

Our main result is the following
\begin{theorem}\label{k-dzet-szer}
Let $S$ be a hyperelliptic surface. Let $L$ be a line bundle of type  $(m,m)$ with $m\geq k+2$ on $S$. Then $L$ is  $k$-jet ample.
\end{theorem}

By the results of M. Mella and M.~Palleschi, see \cite{MP1993}, Theorems 3.2-3.4, we know that $L\equiv (a, b)$ with at least one of the coefficients  strictly smaller than $k+2$ is not $k$-very ample on an arbitrary hyperelliptic surface, in particular it is not $k$-very ample on a hyperelliptic surface of type 1. A line bundle which is not $k$-very ample is not $k$-jet ample. Therefore the line bundle $L\equiv (k+2,k+2)$ is the first natural object of study.

\begin{proof}
We will prove that $L\equiv (k+2,k+2)$ is $k$-jet ample and as a consequence we will get that a line bundle of type $(m,m)$ with $m\geq k+2$ is $k$-jet ample.

Let $r\geq 1$. We have to check that for each choice of distinct points $x_1, \ldots, x_r\in S$ 
the map $$H^0(X,L) \longrightarrow H^0\left(X,L\otimes  \Oo_{X}/(m_{x_1}^{k_1}\otimes\ldots\otimes m_{x_r}^{k_r})\right)$$
is surjective, where $\sum_{i=1}^r k_i=k+1$.

\vskip
 5pt

We consider the standard exact sequence:
\small
$$0\longrightarrow (K_S+L)\otimes m_{x_1}^{k_1}\otimes\ldots\otimes
 m_{x_r}^{k_r}\longrightarrow K_S+L \longrightarrow (K_S+L)\otimes \Oo_{X}/(m_{x_1}^{k_1}\otimes\ldots\otimes m_{x_r}^{k_r})\longrightarrow 0.$$
\normalsize

By the long sequence of
 cohomology, surjectivity of the map
$$H^0(K_S+L) \longrightarrow H^0\left((K_S+L)\otimes  \Oo_{X}/(m_{x_1}^{k_1}\otimes\ldots\otimes m_{x_r}^{k_r})\right)$$  is implied by vanishing of 
 $H^1\left((K_S+L)\otimes m_{x_1}^{k_1}\otimes\ldots\otimes m_{x_r}^{k_r}\right)$.

By the projection formula, we have that
$$H^1 \left((K_S+L)\otimes m_{x_1}^{k_1}\otimes\ldots\otimes m_{x_r}^{k_r}\right)\cong
H^1 \left(\pi^*(K_S+L)-\sum_{i=1}^r k_i E_i\right)\cong$$
$$H^1 \left(K_{\widetilde{S}}-\sum_{i=1}^r E_i+\pi^*L-\sum_{i=1}^r k_i E_i\right)\cong H^1 \left(K_{\widetilde{S}}+\pi^*L-\sum_{i=1}^r (k_i+1) E_i\right).$$

We will show that
  $H^1 \left(K_{\widetilde{S}}+\pi^*L-\sum_{i=1}^r (k_i+1) E_i\right)=0$, using vanishing theorems.

We consider separately case $r=1$, and separately case $r\geq 2$.

\vskip 5pt
First  let $r=1$. 
We show that  $\pi^*L-(k+2)E$ is nef and big, hence by 
 Kawamata-Viehweg vanishing theorem we get that  $H^1 (K_{\widetilde{S}}+\pi^*L-(k+2)E)=0$.

We have that
$$\pi^*L-(k+2)E=\pi^*((k+2,k+2))-(k+2)E=
(k+2)\left(\pi^*(1,1)-E\right).$$

By \cite{Fa2015}, Theorem 3.1,  we know that on a hyperelliptic surface the Seshadri constant of a line bundle of type  $(1,1)$ at an arbitrary point $x$ is at least $1$.
Therefore $$\sup \left\{\varepsilon: \pi^*L-\varepsilon E \text{ is nef} \right\}=
\varepsilon(L,x)=(k+2)\cdot\varepsilon\left((1,1),x\right)\geq k+2,$$
hence the line bundle $\pi^*L-(k+2) E$ is nef. Thus to prove that   $\pi^*L-(k+2) E$ is also big, it is enough to show that $\left(\pi^*L-(k+2)E\right)^2> 0$, which is equivalent to prove that $L^2>(k+2)^2$. The last inequality holds, as $$L^2=(k+2,k+2)^2=2(k+2)\cdot(k+2)=2(k+2)^2.$$

The case $r=1$ is proved.

\vskip 5pt
Now let $r\geq 2$. We will prove that $H^1 \left(K_{\widetilde{S}}+\pi^*L-\sum_{i=1}^r (k_i+1) E_i\right)=0$. The proof will be divided in several cases, depending on the position of points $x_1$, $\ldots$, $x_r$.

Let $k=1$. Generation of $1$-jets is by definition equivalent to $1$-very ampleness. The line bundle $L\equiv (3,3)$ is $1$-very ample on any hyperelliptic surface by \cite{MP1993}, Theorems 3.2-3.4.

Let $k\geq 2$.

If $k_i=0$ for some $i$, then we can consider points $x_1$, $\ldots$, $x_{i-1}$, $x_{i+1}$, $\ldots$, $x_r$. In this case without lose of generality we may take a smaller  $r$. From now on we assume that 
 $k_i\geq 1$ for every $i$. Obviously,  $r\in[2,k+1]$ as $\sum_{i=1}^r k_i=k+1$.

For simplicity, we present a proof for hyperelliptic surfaces of type 1. For surfaces of other types the proof is analogous. The small differences are listed in  Remark \ref{uw_idzie_na_innych_S}.

\vskip 5pt
The proof consists of a few steps which we describe briefly
before we turn to the details.

First, in Case I, we consider a situation where on each singular fibre $A/2$, on each fibre $A$, and on each fibre $B$ there are points $x_i$
with the sum of multiplicities $k_i$ equal to at most~$\frac{k+1}2$.

Then, in Cases II and III, we consider  a situation where there exists a fibre $A/2$, respectively $A$, on which there are some points $x_i$
with the sum of multiplicities $k_i$ greater than $\frac{k+1}2$.

In both cases we have two possibilities: (a) the sum of multiplicities of points lying on any fibre $B$ is smaller than  $\frac{k+1}2$; (b) there exists a fibre $B$ for which the sum of multiplicities of points on this fibre is at least $\frac{k+1}2$. Therefore we divide  Cases  II and III into two subcases: respectively IIa, IIb and IIIa, IIIb.

Finally, in Case IV we consider the situation where some points $x_i$ lie on a fixed fibre $B$ and their sum of multiplicities does not exceed $\frac{k+1}2$, moreover for each fibre $A/2$ and for each fibre $A$ the sum of multiplicities of points lying on this fibre is at most $\frac{k+1}2$. This covers all possibilities.

In all the cases described above we prove that $H^1 \left(K_{\widetilde{S}}+\pi^*L-\sum_{i=1}^r (k_i+1) E_i\right)=0$, using Kawamata-Viehweg vanishing theorem in Cases I and IIIa, and Norimatsu lemma in Cases IIa, IIb, IIIb and IV. In Cases I and IIIa we show that a divisor $M=\pi^*L-\sum_{i=1}^r (k_i+1) E_i$ is big and nef, while in Cases  IIa, IIb, IIIb and IV we define an appropriate SNC divisor $F$ and prove that a divisor $N=M-F$ is ample, using Nakai-Moishezon criterion. 

In all the cases by $C\equiv (\alpha,\beta)$ we denote a reduced irreducible curve on $S$ passing through $x_1$, $\ldots$, $x_r$ with multiplicities respectively $m_1$, $\ldots$, $m_r$, where $m_i\geq 0$ for all $i$, and there exists  $j$ with $m_j>0$. We have $\widetilde{C}=\pi^*C-\sum_{i=1}^r m_iE_i$.

Define $k_i^W=\left\{\begin{array}{l l}k_i & \text{if } x_i\in W \text{ for a reduced fibre } W \text{ of } \Phi \text{ or } \Psi,\\
0&\text{otherwise.}\end{array}\right.$

Let $r_W$ be the number of points from the set $\{x_1, \ldots, x_r\}$ which are contained in $W$.

Let us now move on to considering the cases in more detail.

\vskip 5pt
\textbf{Case I.} For an arbitrary fibre $W$ (where $W=A/2$, or $W=A$, or $W=B$) 
the sum of multiplicities of the points $x_i$ lying on this fibre 
is at most $\frac{k+1}2$, i.e.  $$\sum_{i=1}^{r_W}k_i^W\leq \frac{k+1}2.$$

Since $\sum_{i=1}^{r_W}k_i^W\leq \frac{k+1}2$, in particular we have that  $r_W\leq \frac{k+1}2$.

We will show that the line bundle $M=\pi^*L-\sum_{i=1}^r (k_i+1) E_i$ is big and  nef. 

\begin{lemma}
$M$ is a nef line bundle.
\begin{proof}
 We ask whether  $M\widetilde{C}\geq 0$. We have to check that
$$(\star)=M\widetilde{C}=\left(\pi^*L-\sum_{i=1}^r (k_i+1) E_i\right).\left(\pi^*C-\sum_{i=1}^r m_iE_i\right)  \geq  0$$

Let us consider the following cases:

(1) $C=A/2$, or $C=B$, or  $C=A$. Then for  $i=1, \ldots, r$ we have  $m_i=1$, hence
$$(\star)\geq LC-\sum_{i=1}^{r_W} (k_i^W+1)\geq(k+2)-\sum_{i=1}^{r_W} (k_i^W+1)\geq$$
$$(k+2)-\left(\left(\sum_{i=1}^{r_W} k_i^W\right)+r_W\right) \geq (k+2)-\left(\frac{k+1}2+\frac{k+1}2\right)>0.$$

(2) $C$ is not a fibre. Hence $\alpha>0$ and $\beta>0$.
We prove the inequality $M\widetilde{C}\geq 0$ in Proposition \ref{C_nie_wlokno} at the end of the proof of the main theorem.
\end{proof}
\end{lemma}

\begin{lemma}
$M$ is big.
\begin{proof}
Since $M$  is nef, it is enough to prove that  $M^2>0$. 
As $M^2=2(k+2)^2-\sum_{i=1}^r(k_i+1)^2$, we ask whether
$2(k+2)^2 >  \sum_{i=1}^r(k_i+1)^2$.
It suffices to show that 
$$2(k+2)^2  >  \left(\sum_{i=1}^rk_i^2\right)+2(k+1)+r,$$ where $r$ at the end of the formula is of the greatest possible value $k+1$.
Since $\sum_{i=1}^rk_i^2\leq \left(\sum_{i=1}^rk_i\right)^2=(k+1)^2$, our goal is to prove the inequality 
$$2(k+2)^2> (k+1)^2+3k+3,$$
which is elementary.
\end{proof}
\end{lemma}

\vskip 5pt
\textbf{Case IIa.} There exists a fibre $A/2$ on which there are points $x_1$, $\ldots$, $x_s$
with the sum of multiplicities $\sum_{i=1}^s k_i>\frac{k+1}2$ and on each fibre $B$ we have  $\sum_{i=1}^{r_W} k_i^W<\frac{k+1}2$.

Obviously, for any other fibre $A/2$ and for any fibre $A$ we have  $\sum_{i=1}^{r_W} k_i^W<\frac{k+1}2$.

We write
 $$M=\pi^*L-\sum_{i=1}^r (k_i+1)=\pi^*(A/2)-\sum_{i=1}^s E_i+\pi^*((k+1,k+2))-\sum_{i=1}^s k_i E_i-\sum_{i=s+1}^r (k_i+1) E_i,$$
so $M=N+F$, where $N=\pi^*((k+1,k+2))-\sum_{i=1}^s k_i E_i-\sum_{i=s+1}^r (k_i+1) E_i$ and $F=\pi^*(A/2)-\sum_{i=1}^s E_i=\widetilde{A/2}$. Clearly $F$ is a smooth and reduced divisor, hence $F$ is an SNC divisor. It remains to show that $N$ is ample.

\begin{lemma}
$N$ is ample.
\begin{proof}
We check that $N^2>0$, and $N\widetilde{C}>0$. 

Let us estimate $N^2$
$$N^2=2(k+1)(k+2)-\sum_{i=1}^r k_i^2 -2\sum_{i=s+1}^r k_i-(r-s) \geq k^2+2k+1+s>0,$$
as  $\sum_{i=s+1}^r k_i<\frac{k+1}2$, $\sum_{i=1}^r k_i^2\leq \left(\sum_{i=1}^r k_i\right)^2=(k+1)^2$ and $r\leq k+1$.

\vskip 5pt

Now we check whether
$$(\star)=N\widetilde{C}=\left(\pi^*((k+1,k+2))-\sum_{i=1}^s k_i E_i-\sum_{i=s+1}^r (k_i+1) E_i\right).\left(\pi^*C-\sum_{i=1}^r
 m_iE_i\right) >  0 $$

We consider the following cases:

(1) $C$ is the fibre $A/2$ for which $\sum_{i=1}^s k_i>\frac{k+1}2$. Then $m_i=1$ for all $i$, and
$$(\star)=k+2-\sum_{i=1}^s k_i \geq  k+2-(k+1)=1>0,$$%to jest te¿ ten przypadek, ¿e wszystkie x_i le¿¹ na A_2
as $\sum_{i=1}^s k_i\leq \sum_{i=1}^r k_i=k+1$.

(2) $C$ is a different fibre $A/2$, or $C=A$, or $C=B$. 
All $m_i=1$, hence
$$(\star)\geq k+1-\sum_{i=1}^{r_W} (k_i^W+1)  > k+1-\left(\frac{k+1}2+\frac{k+1}2\right)=0,$$
as $\sum_{i=1}^{r_W} k_i^W< \frac{k+1}2$ and $r_W < \frac{k+1}2$.

(3) $C$ is not a fibre --- see Proposition \ref{C_nie_wlokno}.
\end{proof}
\end{lemma}

\vskip 5pt
\textbf{Case IIb.} Points $x_1$, $\ldots$, $x_s$ lie on a fixed fibre $A/2$ with $\sum_{i=1}^s k_i>\frac{k+1}2$, moreover there exists a fibre $B$ such that $x_s$, $\ldots$, $x_{t}$ lie on $B$ with $\sum_{i=s}^{t} k_i\geq\frac{k+1}2$.

We define $F:=\widetilde{A/2}+\widetilde{B}=\widetilde{A/2+B}$. Clearly $F$ is an SNC divisor. We define $N:=M-F=\pi^*((k+1,k+1))+E_s-\sum_{i=1}^{t} k_i E_i-\sum_{i=t+1}^r (k_i+1) E_i$. It remains to check that $N$ is ample.

\begin{lemma}
$N$ is ample.
\begin{proof}
Analogously to Case IIa, we show that $N^2>0$, and that $N\widetilde{C}>0$.

We estimate $N^2$ from below, using inequalities $\sum_{i=t+1}^{r} k_i\leq\frac{k+1}2$, $\sum_{i=1}^{r} k_i^2\leq(k+1)^2$ and $r\leq k+1$.
$$N^2=2k^2+4k+1-\sum_{i=1}^r k_i^2-2\sum_{i=t+1}^r k_i-(r-t)\geq$$
$$2k^2+4k+1-(k+1)^2-2\cdot
 \frac{k+1}2-(k+1-t)=k^2+t-2.$$
Hence $N^2>0$, because $k\geq 2$.

Now we check that
$$(\star)=N\widetilde{C} >  0$$

We consider the following cases:

(1)  $C$ is the fibre $A/2$ for which $\sum_{i=1}^s k_i>\frac{k+1}2$. Then
$$(\star)=k+1+1-\sum_{i=1}^s k_i \geq k+2-(k+1)=1>0,$$
as $\sum_{i=1}^s k_i\leq k+1$.

(2) $C$ is a different fibre $A/2$, or $C=A$. Then
$$(\star)\geq k+1-\sum_{i=1}^{r_W} \left(k_i^W+1\right)  > k+1-\left(\frac{k+1}2+\frac{k+1}2\right)=0,$$
as $\sum_{i=1}^{r_W} k_i^W< \frac{k+1}2$ and $r_W < \frac{k+1}2$.

(3)  $C$ is the fibre $B$ for which $\sum_{i=s}^{t} k_i\geq\frac{k+1}2$. Then
$$(\star)=(k+1)+1-\sum_{i=s}^{t} k_i \geq k+2-(k+1)=1>0.$$

(4) $C$ is a different fibre $B$. Then
$$(\star)\geq (k+1)-\sum_{i=1}^{r_W} (k_i^W+1) > (k+1)-\left(\frac{k+1}2+\frac{k+1}2\right)=0,$$
as $\sum_{i=1}^{r_W} k_i^W<\frac{k+1}2$ and $r_W < \frac{k+1}2$.

(5) $C$ is not a fibre --- see Proposition \ref{C_nie_wlokno}.
\end{proof}
\end{lemma}

\vskip 5pt
\textbf{Case IIIa.} Points $x_1$, $\ldots$, $x_s$ lie on a fixed fibre  $A$ with $\sum_{i=1}^s k_i>\frac{k+1}2$, and for each fibre  $B$ we have $\sum_{i=1}^{r_W} k_i^W<\frac{k+1}2$.

Obviously in this case, for any other fibre $A$ and for any fibre  $A/2$ we have $\sum_{i=1}^{r_W} k_i^W<\frac{k+1}2$.

Let $M=\pi^*L-\sum_{i=1}^r (k_i+1) E_i$. We have already showed in Case I that $M^2>0$. It remains to prove that  $M$ jest nef.

\begin{lemma}
$M$ is nef.
\begin{proof}
We ask whether 
$$(\star)=M\widetilde{C} \geq   0$$

We consider the following cases:

(1) $C$ is the fibre $A$ for which $\sum_{i=1}^s
 k_i>\frac{k+1}2$. Then
$$(\star)=2(k+2)-\sum_{i=1}^{s}(k_i+1)\geq 2(k+2)-\sum_{i=1}^{r} k_i - s\geq$$  $$2(k+2)-(k+1)
 - (k+1)=2>0.$$

(2) $C$ is a different fibre $A$, or $C=A/2$, or  $C=B$. Then
$$(\star)\geq k+2-\sum_{i=1}^{r_W} (k_i^W+1)\geq k+2-\left(\frac{k+1}2+\frac{k+1}2\right)=1> 0.$$

(3) $C$ is not a fibre --- see Proposition \ref{C_nie_wlokno}.
\end{proof}
\end{lemma}

\vskip 5pt
\textbf{Case IIIb.} Points $x_1$, $\ldots$, $x_s$ lie on a fixed fibre  $A$ with $\sum_{i=1}^s k_i>\frac{k+1}2$, and there exists a fibre   $B$, such that  $x_s$, $\ldots$, $x_{t}$ lie on $B$ and $\sum_{i=s}^{t} k_i\geq\frac{k+1}2$.

We define $F:=\widetilde{B}=\pi^*B-\sum_{i=s}^{t}E_i$. Clearly $F$ is an SNC divisor. 
Let $N:=M-F=\pi^*((k+2,k+1))-\sum_{i=1}^{s-1} (k_i+1)E_i-\sum_{i=s}^{t} k_i E_i
-\sum_{i=t+1}^{r} (k_i+1) E_i$. It remains to check that  $N$ is ample.

\begin{lemma}
$N$ is ample.
\begin{proof}
Let us estimate $N^2$:
$$N^2=2(k+2)(k+1)-\sum_{i=1}^{s-1} (k_i+1)^2-\sum_{i=s}^{t} k_i^2-\sum_{i=t+1}^{r} (k_i+1)^2=$$
$$2(k+2)(k+1)-\sum_{i=1}^{r} k_i^2-2\left(\sum_{i=1}^{s-1} k_i-\sum_{i=t+1}^{r} k_i\right)-(s-1)-(r-t)\geq$$
$$2k^2+6k+4-(k+1)^2-2(k+1)-s+1-r+t\geq$$
$$k^2+2k+2-(k+1)-(k+1)+t=k^2+t>0.$$

Now we check that
$$(\star)=N\widetilde{C} >  0$$

Let us consider the following cases:

(1) $C$ is the fibre  $A$ for which $\sum_{i=1}^s k_i>\frac{k+1}2$. Then
$$(\star)=2(k+1)-\sum_{i=1}^{s-1} (k_i+1)-k_s \geq  2(k+1)-\sum_{i=1}^{s} k_i-(s-1)\geq$$
$$2(k+1)-(k+1)-(k+1)+1=1>0.$$

(2)  $C$ is the fibre  $B$ for which $\sum_{i=s}^{t} k_i\geq\frac{k+1}2$. Then
$$(\star)=k+2-\sum_{i=s}^{t} k_i \geq k+2-(k+1)=1>0.$$

(3) $C$ is a different fibre $A$ or a different fibre $B$, or $C=A/2$. Then
$$(\star)\geq k+1-\sum_{i=1}^{r_W} (k_i^W+1)  > k+1-\left(\frac{k+1}2+\frac{k+1}2\right)=0.$$

(4) $C$ is not a fibre --- see Proposition \ref{C_nie_wlokno}.
\end{proof}
\end{lemma}

\vskip 5pt
\textbf{Case IV.} Points $x_1$, $\ldots$, $x_s$ lie on a fixed fibre  $B$ with $\sum_{i=1}^s k_i>\frac{k+1}2$, and for each fibre 
$A/2$, and for each fibre  $A$ the sum of multiplicities of the points lying on this fibre does not exceed $\frac{k+1}2$.

We define $F:=\widetilde{B}=\pi^*B-\sum_{i=1}^{s}E_i$. Of course $F$ is an SNC divisor. We define $N:=M-F=\pi^*((k+2,k+1))-\sum_{i=1}^{s} k_i E_i
-\sum_{i=s+1}^{r} (k_i+1) E_i$. It remains to prove that  $N$ is ample.

\begin{lemma}
$N$ is ample.
\begin{proof}
Analogously to Case  IIa,
$$N^2= k^2+2k+1+s>0.$$

\vskip 5pt

We have to check that
$$(\star)=N\widetilde{C} >  0 $$

Let us consider the following cases:

(1) $C$ is the fibre $B$ for which $\sum_{i=1}^s k_i>\frac{k+1}2$. Then
$$(\star)=k+2-\sum_{i=1}^s k_i \geq k+2-(k+1)=1>0.$$

(2) $C$ is a different fibre $B$ or  $C=A$. Then
$$(\star)\geq k+2-\sum_{i=1}^{r_W} (k_i^W+1) \geq
 k+2-\left(\frac{k+1}2+\frac{k+1}2\right)=1>0.$$

(3) $C=A/2$.

If $r_W<\frac{k+1}2$, then
$$(\star)=k+1-\sum_{i=1}^{r_W} (k_i^W+1)>k+1-\left(\frac{k+1}2+\frac{k+1}2\right)=0.$$

Otherwise  $r_W=\frac{k+1}2$. Hence the situation is as follows: 
 $s$ points whose sum of multiplicities is greater than  $\frac{k+1}2$ lie on the fixed fibre  $B$, and there exists a fibre 
 $A/2$ with $\frac{k+1}2$ points whose sum of  multiplicities does not exceed $\frac{k+1}2$, which implies that all the multiplicities of points lying on $A/2$ equal  $1$. Hence  $x_1$, $\ldots$, $x_s$ lie on $B$ and $x_s$,~$\ldots$, $x_r$ lie on  $A/2$. Therefore
$$(\star)=k+1-k_s-\sum_{i=s+1}^r (k_i+1)= k+1-1-2(r-s)=k+2-2\cdot\frac{k+1}2=1>0$$

(4) $C$ is not a fibre --- see Proposition \ref{C_nie_wlokno}.
\end{proof}
\end{lemma}

We have considered all the possible positions of  $x_1$, $\ldots$, $x_r$.
Hence the proof will be completed if we show that for a curve  $C$ not being a fibre the inequality respectively $M\widetilde{C}\geq 0$ or $N\widetilde{C}>0$ holds. We prove this fact in the proposition below. 

\begin{proposition}\label{C_nie_wlokno}
We have the following inequalities for a curve $C$ which is not a fibre:
\begin{itemize}
\item $M\widetilde{C}\geq 0 \text{ in Cases I and IIIa,}$
\item $N\widetilde{C}> 0 \text{ in Cases IIa, IIb, IIIb and IV.}$
\end{itemize}

\begin{proof}
By  assumption $C\equiv(\alpha,\beta)$ with $\alpha>0$ and $\beta>0$.
We have to prove that
\begin{itemize}
\item $\displaystyle\bigg(\bigg(\sum_{i=1}^rk_i\bigg)+1\bigg)(\alpha+\beta) - \sum_{i=1}^r(k_i+1)m_i \geq 0$ \  in Cases I and IIIa;
\item $\displaystyle\bigg(\sum_{i=1}^rk_i\bigg)(\alpha+\beta)+\alpha - \sum_{i=1}^sk_im_i-\sum_{i=s+1}^r(k_i+1)m_i > 0$ \  in Case IIa;
\item $\displaystyle\bigg(\sum_{i=1}^rk_i\bigg)(\alpha+\beta)+m_s - \sum_{i=1}^{t}k_im_i -\sum_{i=t+1}^r(k_i+1)m_i> 0$ \ in Case IIb;
\item $\displaystyle\bigg(\sum_{i=1}^rk_i\bigg)(\alpha+\beta)+\beta - \sum_{i=1}^{s-1}(k_i+1)m_i- \sum_{i=s}^{t}k_im_i- \sum_{i=t+1}^r(k_i+1)m_i > 0$ \ in Case IIIb;
\item $\displaystyle\bigg(\sum_{i=1}^rk_i\bigg)(\alpha+\beta)+\beta - \sum_{i=1}^sk_im_i-\sum_{i=s+1}^r(k_i+1)m_i > 0$ \
in Case IV.
\end{itemize}

\vskip 5pt
Observe that in all the situations above, Proposition \ref{C_nie_wlokno} will be proved if we show that
$$(\star) \qquad \bigg(\sum_{i=1}^r k_i\bigg)(\alpha+\beta)\geq\sum_{i=1}^r (k_i+1)m_i.$$
\vskip 5pt

Let $D\equiv (4,4)$. Since $h^0(D)=4\cdot 4=16$, for an arbitrary point $x$ there exists a divisor $D_x\in |D|$ such that $\mult_xD_x=5$ (vanishing up to order 5 imposes 15 conditions). Hence there are two possibilities:
either $\alpha\leq 4$ and $\beta\leq 4$, and then  $C$ and $D_x$ may have a common component $C$ (the curve $C$ is irreducible); or 
$\alpha> 4$ or $\beta> 4$, and then by B{\'e}zout's Theorem $$4(\alpha+\beta)=(\alpha,\beta).(4,4)=CD=CD_x\geq \mult_xC\cdot \mult_xD_x\geq 5m_i.$$
Proof of the proposition in each case will be completed in Lemma \ref{alfa>4} and Lemma \ref{alfa<=4}.
\end{proof}
\end{proposition}

\begin{lemma}\label{alfa>4}
Proposition \ref{C_nie_wlokno} holds if $\alpha> 4$ or $\beta> 4$.
\end{lemma}
\begin{proof}
We begin with a useful observation:
\begin{observation}\label{redukcja_k_i=1}
Let $S$ be a hyperelliptic surface of any type, let  $C\equiv (\alpha,\beta)$, where $\alpha> 4$ or $\beta> 4$. To prove 
 $(\star)$, it suffices to prove the inequality
$$(\star\star) \qquad r(\alpha+\beta)\geq 2\sum_{i=1}^r m_i.$$

\begin{proof} We have already observed that by 
 B{\'e}zout's Theorem, for every $i\in\{1,\ldots, r\}$ we have that
$\alpha+\beta\geq \frac 5 4m_i\geq m_i,$
hence
 $$(k_i-1)(\alpha+\beta)\geq (k_i-1)m_i.$$
Summing up these $r$ inequalities with inequality
$(\star\star)$
we obtain the inequality $(\star)$.
\end{proof}
\end{observation}

Now we will show that if  $\alpha> 4$ or $\beta> 4$ then $(\star\star)$ is satisfied for  $r\geq 3$.  
Let us denote $i$-th inequality in the inequality  $(\star\star)$ by $(\star\star_i)$, i.e.
$$(\star\star_i) \qquad (\alpha+\beta)\geq 2 m_i.$$

The inequality ($\star\star_i$) is satisfied for $m_i< 4$. Indeed,  $C$ is not a fibre and by assumption we have
$\alpha\geq 5$ or $\beta\geq 5$, hence $\alpha+\beta\geq 5+1=6\geq 2m_i$ for  $m_i\leq 3$. Therefore we may assume that $m_i\geq 4$.

We delete from $(\star\star)$ all the inequalities  $(\star\star_i)$ with $m_i< 4$  and consider a modified inequality ($\star\star$), possibly with a smaller number of points $r$. It may even happen that $r<2$ in the modified ($\star\star$).

Now we will prove $(\star\star)$ assuming that $m_i\geq 4$ for all  $i$, and $r\geq 3$. Equivalently, we want to show that
$$r^2(\alpha+\beta)^2 \geq  4\left(\sum_{i=1}^r m_i\right)^2$$
By inequality between means, it is enough to check that
$$r^2(\alpha+\beta)^2 \geq  4r\sum_{i=1}^r m_i^2$$
It suffices to check that
$$r\left(2\alpha\beta\right) \geq  2\sum_{i=1}^r m_i^2$$
By the genus formula $2\alpha\beta\geq\sum_{i=1}^r m_i^2-\sum_{i=1}^r m_i$ (see Proposition \ref{gen_form} and Observation \ref{gen_form_hyperell}), hence it is enough to prove that 
$$r\left(\sum_{i=1}^r m_i^2-\sum_{i=1}^r m_i\right)\geq  2\sum_{i=1}^r m_i^2$$
$$(r-2)\left(\sum_{i=1}^r m_i^2\right)-r\sum_{i=1}^r m_i\geq   0$$
We assume that $m_i\geq 4$, so $\sum_{i=1}^r m_i^2\geq 4\sum_{i=1}^r m_i$. Hence it is enough to show that
$$4(r-2)\left(\sum_{i=1}^r m_i\right)-r\sum_{i=1}^r m_i \geq   0$$
$$(3r-8)\cdot(4r) \geq  0$$
which is obviously true for $r\geq 3$.

To finish the proof of Lemma, it remains to check that the assertion holds for $r<3$.
If  $r=0$, then every inequality $(\star\star_i)$ is satisfied, which together with Observation~\ref{redukcja_k_i=1} completes the proof of the lemma. Hence we have to consider cases
 $r=1$ (when in the inequality $(\star\star)$ all but one inequalities $(\star\star_i)$ hold), and $r=2$.

\vskip 5pt

Let $r=2$. We prove Lemma in Cases I and IIIa, and separately in all the remaining cases. 

Let us consider Cases  I and IIIa. We have to prove the inequality
$$(k_1+k_2+1)(\alpha+\beta) - (k_1+1)m_1 - (k_2+1)m_2 \geq 0$$

Analogously to Observation \ref{redukcja_k_i=1} it suffices to prove the inequality $(1+1)(\alpha+\beta)+ (\alpha+\beta)\geq 2m_1+2m_2$. Indeed, adding two fulfilled inequalities of the form $(k_i-1)(\alpha+\beta)\geq (k_i-1)m_i$, we obtain the assertion.

Therefore we have to check that
$$3(\alpha+\beta) \geq   2m_1+2m_2.$$

In the linear system of divisor $D\equiv (4,4)$, for each points $x_1$,~$x_2$, there exists such a divisor 
 $D_{x_1,x_2}\equiv D$ that its multiplicity at  $x_1$ equals $4$, and multiplicity at  $x_2$ equals $2$.

Since $\alpha> 4$ or  $\beta> 4$ and the curve  $C$ is irreducible,  $C$ is not a component of $D$. Therefore by  B{\'e}zout's Theorem we have 
\begin{equation}\label{dla_2_punktow}
4(\alpha+\beta)=CD=CD_{x_1,x_2} \geq 4m_1+2m_2.
\end{equation}
Analogously
$$4(\alpha+\beta) \geq 2m_1+4m_2.$$
Summing up two inequalities above, we obtain
$$8(\alpha+\beta)\geq 6m_1+6m_2.$$
Hence
$$3(\alpha+\beta)\geq 
\frac{9}4m_1+\frac{9}4m_2\geq 2m_1+2m_2,$$
and the assertion is proved.

\vskip 5pt
Now let us consider Cases IIa, IIb, IIIb and IV.
There are two possibilities:

(a) One of the points $x_1$, $x_2$ (without loss of generality $x_2$) lies  respectively: on the fixed fibre  $A/2$ in Case IIa, on the intersection of the fixed fibres  $A/2$ and $B$ in Case IIb, on the intersection of the fixed fibres  $A$ and $B$ in Case  IIIb, on the fixed fibre $B$ in Case IV. Then the desired inequalities $N\widetilde{C}>0$ are implied by the inequality
$$(k_1+k_2)(\alpha+\beta) \geq (k_1+1)m_1+k_2m_2.$$
By observation analogous to Observation \ref{redukcja_k_i=1} it suffices to show that
$$2(\alpha+\beta) \geq   2m_1+m_2.$$
The assertion holds by  inequality  \eqref{dla_2_punktow}.

(b) None of the points  $x_1$, $x_2$ lies on the fixed fibre respectively $A/2$, $A/2$ intersected with $B$, $A$ intersected with $B$, and $B$. By assumption of Case respectively IIa, IIb, IIIb and IV, at the beginning there was a point $x_3$ on those fixed fibres or on the intersection of the fixed fibres, and the inequality ($\star\star_3$) is satisfied. We restore this inequality, and we want to prove that 
$$(k_1+k_2+k_3)(\alpha+\beta) \geq (k_1+1)m_1+(k_2+1)m_2+k_3m_3$$
By observation analogous to Observation \ref{redukcja_k_i=1}, it is enough to prove that
$$3(\alpha+\beta)\geq   2m_1+2m_2+m_3$$

In the linear system of divisor  $D\equiv (4,4)$ for each points $x_1$,  $x_2$, $x_3$  there exists such a divisor 
 $D_{x_1,x_2,x_3}\equiv D$ that its multiplicity at  $x_1$ equals $3$,  its multiplicity at  $x_2$ equals $3$, and its multiplicity at  $x_3$ equals $2$.

By B{\'e}zout's Theorem we obtain
\begin{equation}\label{dla_3_punktow}
4(\alpha+\beta)=CD=CD_{x_1,x_2,x_3} \geq \sum_{i=1}^3 \mult_{x_i}C\cdot \mult_{x_i}D_{x_1,x_2,x_3} \geq 3m_1+3m_2+2m_3.
\end{equation}
Hence
$$3(\alpha+\beta)\geq \frac{9}4m_1+\frac{9}4m_2+\frac{6}4m_3\geq
  2m_1+2m_2+m_3,$$
and we are done.

\vskip 5pt

Now let $r=1$. At the beginning the number of points $r$ was at least $2$, hence while deleting from the inequality  ($\star\star$) the inequalities ($\star\star_i$) with $m_i\leq 3$, we deleted all by one inequalities.

We restore one of the deleted inequalities ($\star\star_i$): 
 an arbitrary one in Cases I and IIIa;  in Cases  IIa, IIb, IIIb and IV --- the inequality corresponding to  $x_i$ lying respectively on the fixed fibre $A/2$, on the intersection of the fixed fibres $A/2$ and $B$, on the intersection of the fixed fibres $A$ and $B$, on the fixed fibre $B$ (if $x_1$ is not in such a position; an arbitrary inequality ($\star\star_i$) otherwise).
We obtained the inequality with $r=2$ which was already proved in subcase~(a).
\end{proof}

\begin{lemma}\label{alfa<=4}
Proposition \ref{C_nie_wlokno} holds if $\alpha \leq 4$ and $\beta\leq 4$.
\end{lemma}
\begin{proof}
We denote 
$$(\star_i) \qquad (\alpha+\beta) k_i\geq (k_i+1)m_i.$$

Observe the following property:
\begin{remark}\label{ogr_krot_m_i} 
For an arbitrary  $m_i$ the inequality $(\star_i)$ is satisfied if  $m_i\leq \frac 1 2(\alpha+\beta)$. Indeed,  $2m_i\geq\left(1+\frac 1 {k_i}\right)m_i$ for all $k_i$.

In particular, if  $m_i=0$ or  $m_i=1$, then the inequality  $(\star_i)$ is satisfied, since  $\alpha\geq 1$ and $\beta\geq 1$.
\end{remark}

The multiplicities of $C$ at  $x_1$, $\dots$, $x_r$ satisfy genus formula, i.e. $2\alpha\beta\geq \sum_{i=1}^r m_i^2 - \sum_{i=1}^r m_i$ (see Proposition \ref{gen_form} and Observation \ref{gen_form_hyperell}). In particular, for any  $x_i$ we have an upper bound $2\alpha\beta\geq m_i^2 -  m_i$.

If the inequality $(\star_i)$ holds for some $m_i$, then it holds also for all multiplicities $n_i<m_i$. After renumbering the points we assume that the  multiplicities are decreasing: $m_1\geq m_2\geq\ldots\geq m_r$.

\vskip 5pt
In the table below, for every curve $C\equiv(\alpha,\beta)$ with $0<\alpha\leq 4$ and $0<\beta\leq 4$,  we present the following quantities:\\
-- an upper bound for the maximal possible multiplicity $m_i$ at any point  $x_i$, obtained from genus formula,\\
-- an upper bound for the multiplicity $m_i$, for which by Remark \ref{ogr_krot_m_i} the inequality  ($\star_i$) holds,\\
-- all possible values of $m_1$ greater than the number from the previous column, and for each such $m_1$ the greatest possible value of $m_2$ and $m_i$ for $i>2$ (or an upper bound in the latter case), all obtained from genus formula.
$$
\begin{array}{c||c||c|c||c|c|c}
\text{Case no.}&C\equiv(\alpha,\beta)&\max m_i& \max m_i\text{ such that } (\star_i) \text{ holds}&m_1&m_2&m_i \text{ for } i>2\\
\hline\hline
\text{i}&(4,4)&6&4&\scriptsize{\begin{array}{c}
6\\5\end{array}}&\scriptsize{\begin{array}{c}
2\\4\end{array}}& 1\\
\hline
\text{ii}&(4,3) \text{ or } (3,4)&5&3&\scriptsize{\begin{array}{c}
5\\4\end{array}}&\scriptsize{\begin{array}{c}
2\\3\end{array}}&\leq 2\\
\hline
\text{iii}&(4,2)\text{ or } (2,4)&4&3&4&2&\leq 2\\
\hline
\text{iv}&(4,1)\text{ or } (1,4)&3&2&3&2& 1\\
\hline
\text{v}&(3,3)&4&3&4&3& 1\\
\hline
\text{vi}&(3,2)\text{ or } (2,3)&4&2&\scriptsize{\begin{array}{c}
4\\3\end{array}}&\scriptsize{\begin{array}{c}
1\\3\end{array}}& 1\\
\hline
\text{vii}&(3,1)\text{ or } (1,3)&3&2&3&1& 1\\
\hline
\text{viii}&(2,2)&3&2&3&2& 1\\
\hline
\text{ix}&(2,1)\text{ or }(1,2)&2&1&2&2& 1\\
\hline
\text{x}&(1,1)&2&1&2&1& 1\\
\hline
\end{array}
$$

Note that in each case (i)-(x), the inequality $(\star_i)$ holds for each  multiplicity $m_i$ where $i>2$. Therefore we would be done if we could prove $(\star)$ for $r=2$ and points $x_1$, $x_2$ with multiplicities $m_1$, $m_2$ listed in the table.

Simple computations give us that 
\begin{remark}\label{lem_r=2}\textcolor{white}{a}
\begin{enumerate}
\item For  $r=2$, if $m_1+m_2\leq\alpha+\beta$, then $(\star)$ holds.
\item Moreover for $r=2$, if $ m_1+m_2\leq \alpha+\beta+1$, then the inequality from Cases I and~IIIa holds, i.e. $(\alpha+\beta) \left(\sum_{i=1}^r k_i+1\right)\geq \sum_{i=1}^r (k_i+1)m_i$. 
\end{enumerate}
\end{remark}

Let us consider the cases listed in the table. In cases  (ii), (iii), (iv), (vii) the inequality~($\star$) holds by  Remark \ref{lem_r=2}(2). In cases  (i), (v), (vi), (viii), (ix) we repeat the reasoning with restoring a certain inequality ($\star_3$), if it is needed. Case (x) is the most subtle --- in some subcases we have to restore two inequalities ($\star_3$) and ($\star_4$). 

We explain the cases (i) and (x) in detail.

\vskip 5pt

Case (i): $C\equiv (4,4)$.

It is enough to prove $(\star)$ in two subcases: $r=2$, $m_1=6$, $m_2=2$, and $r=2$, $m_1=5$, $m_2=4$.
In the first subcase the we get the assertion by Remark \ref{lem_r=2}(1). 
In the second subcase, in some situations we have to restore the fulfilled inequality $(\star_3)$ to prove $(\star)$. Our resoning differs with respect to cases of the main theorem:\\
-- In Cases I and IIIa the inequality $(\star)$ holds by Remark \ref{lem_r=2}(2).\\
-- In Cases IIa, IIIb, IV there are two possibilities:

\textbullet \ One of the points $x_1$, $x_2$ (without loss of generality $x_2$) lies  respectively: on the fixed fibre  $A/2$ in Case IIa,  on the intersection of the fixed fibres  $A$ and $B$ in Case  IIIb, on the fixed fibre $B$ in Case IV. Then it is enough to prove that
$$(\alpha+\beta)(k_1+k_2)+\min\{\alpha,\beta\} >   (k_1+1)m_1+k_2m_2$$
$$8(k_1+k_2)+4 >   5(k_1+1)+4k_2$$
$$3k_1+4k_2 >  1$$
The last inequality holds as  $k_i\geq 1$ for all $i$.

\textbullet \ None of the points  $x_1$, $x_2$ lies on the fixed fibre respectively $A/2$, $A$ intersected with $B$, and $B$. By assumption of Cases respectively IIa,  IIIb and IV, at the beginning there was a point $x_3$ on the fixed fibre or on the intersection of the fixed fibres, and the maximal possible by genus formula multiplicity $m_3$ is equal to $1$. We restore the inequality ($\star_3$), and we want to prove that 
$$(\alpha+\beta)(k_1+k_2+k_3)+\min\{\alpha,\beta\} >   (k_1+1)m_1+(k_2+1)m_2+k_3m_3$$
$$3k_1+4k_2+7k_3 >   5$$
The inequality holds.\\
-- In Case IIb there are also two possibilities:

\textbullet \ One of the points $x_1$, $x_2$ is the intersection  point of the fixed fibres  $A/2$ and $B$. Then the inequality
$$(\alpha+\beta)(k_1+k_2)+m_2 >   (k_1+1)m_1+k_2m_2$$
$$3k_1+4k_2 >   1$$
obviously holds.

\textbullet \ None of the points  $x_1$, $x_2$ lies on the intersection of the fixed fibres $A/2$ and $B$. Then by assumption  of Case IIb there was a point  $x_3$ with $m_3=1$ on this intersection. We restore the inequality ($\star_3$) and have to check that the following inequality is satisfied 
$$(\alpha+\beta)(k_1+k_2+k_3)+m_3 >   (k_1+1)m_1+(k_2+1)m_2+k_3m_3$$
$$3k_1+4k_2+7k_3 >   8,$$
which is true.

\vskip 5pt
Case (x): $C\equiv (1,1)$. 

We consider a situation $r=2$, $m_1=2$, $m_2=1$. If needed, we restore one or two inequalities ($\star_i$) and prove ($\star$) for $r=3$, $m_1=2$, $m_2=1$, $m_3=1$ or for $r=4$, $m_1=2$, $m_2=1$, $m_3=1$, $m_4=1$. Here come the details:\\
-- In Cases I and IIIa the inequality $(\star)$ holds by Remark \ref{lem_r=2}(2).\\
-- In Cases IIa, IIIb, IV there are two possibilities:

\textbullet \ One of the points $x_1$, $x_2$ (say $x_2$) lies  respectively: on the fixed fibre  $A/2$ in Case IIa,  on the intersection of the fixed fibres  $A$ and $B$ in Case  IIIb, on the fixed fibre $B$ in Case IV. Then it suffices to prove that
$$(\alpha+\beta)(k_1+k_2)+\min\{\alpha,\beta\} >  (k_1+1)m_1+k_2m_2$$
which gives
$$k_2 >   1$$
If $k_2=1$, then respectively: on the fixed fibre  $A/2$, on the fixed fibre $B$ but outside the intersection of the fixed fibres  $A$ and $B$ (otherwise $k=1$, but we have excluded such a situation at the beginning of main theorem's proof), on the fixed fibre $B$, there was originally at least one more point, as for points on the fibre $\sum k_i^W>\frac{k+1}2$. If $x_1$ is the described point, it suffices to prove that
$$(\alpha+\beta)(k_1+k_2)+\min\{\alpha,\beta\} > k_1m_1+k_2m_2$$
$$k_2+1 > 0$$
The inequality holds. If the described point is different from $x_1$, then it has multiplicity  $m_3=1$. We have to check that
$$(\alpha+\beta)(k_1+k_2+k_3)+\min\{\alpha,\beta\} > (k_1+1)m_1+k_2m_2+k_3m_3$$
$$k_2+k_3 >  1$$
The inequality holds.

\textbullet \ None of the points  $x_1$, $x_2$ lies on the fixed fibre respectively $A/2$, $A$ intersected with $B$, and $B$. By assumption of Case respectively IIa,  IIIb and IV, at the beginning there was a point $x_3$ on the fixed fibre or on the intersection of the fixed fibres, and the maximal possible by genus formula multiplicity $m_3$ is equal to $1$. We restore the inequality ($\star_3$), and we have to check that 
$$(\alpha+\beta)(k_1+k_2+k_3)+\min\{\alpha,\beta\} >  (k_1+1)m_1+(k_2+1)m_2+k_3m_3$$
$$k_2+k_3 >  2$$
If $k_3=1$, then by assumption of Case IIa, IIIb, IV respectively on the fixed fibre $A/2$, on the fixed fibre  $B$ but outside the intersection of $A$ and $B$, on the fixed fibre $B$, there is a point $x_4$. In Case IIIb $x_4$ may be equal to one of the points $x_1$, $x_2$ --- we obtain an inequality which was already proved. Otherwise in Case IIIb, and in Cases IIa and IV $x_4$ is different from  $x_1$, $x_2$, $x_3$ and $m_4=1$. We have to prove that
$$(\alpha+\beta)(k_1+k_2+k_3+k_4)+m_3 >  (k_1+1)m_1+(k_2+1)m_2+k_3m_3+k_4m_4$$
$$k_2+k_3+k_4>2$$
The inequality holds.\\
-- In Case IIb there are two possibilities as well:

\textbullet \ One of the points $x_1$, $x_2$ (say $x_2$) is  the intersection point of the fixed fibres  $A/2$ and $B$. Then we have to prove that
$$(\alpha+\beta)(k_1+k_2)+m_2 >   (k_1+1)m_1+k_2m_2$$
$$k_2 >   1$$
If $k_2=1$, then by assumption of Case IIb there is at least one more point $x_3$ on the fixed fibre  $A/2$. If  $x_3=x_1$ then it is enough to check that
$$(\alpha+\beta)(k_1+k_2)+m_2 >   k_1m_1+k_2m_2$$
$$k_2+1 >   0$$
The inequality holds. Otherwise  $m_3=1$, and it suffices to show that
$$(\alpha+\beta)(k_1+k_2+k_3)+m_2 >   (k_1+1)m_1+k_2m_2+k_3m_3$$
$$k_2+k_3 >   1$$
The inequality holds.

\textbullet \ None of the points $x_1$, $x_2$ in the intersection point of the fixed fibres  $A/2$ and $B$. Then on this intersection there is a point  $x_3$ with $m_3=1$. Hence it is enough to prove that 
$$(\alpha+\beta)(k_1+k_2+k_3)+m_3 >  (k_1+1)m_1+(k_2+1)m_2+k_3m_3$$
$$k_2+k_3 > 2$$
If $k_3=1$ then there is one more point $x_4$ on the fixed fibre $A/2$. If $x_4$ is equal to  $x_1$ or $x_2$ then we get the already proved inequality
$$(\alpha+\beta)(k_1+k_2+k_3)+m_3 >   (k_1+1)m_1+k_2m_2+k_3m_3$$
If $x_4$ is different from  $x_1$, $x_2$, $x_3$ then $m_4=1$, and we obtain the inequality
$$(\alpha+\beta)(k_1+k_2+k_3+k_4)+m_3 >  (k_1+1)m_1+(k_2+1)m_2+k_3m_3+k_4m_4$$
which has been proved as well. 
\end{proof}
The proof of the Lemma \ref{alfa>4} and Lemma \ref{alfa<=4} is completed, hence the main theorem has been proved. 
\end{proof}

\vskip 5pt

We have presented the detailed proof for hyperelliptic surfaces of type 1. Below we list the small differences which occur in the proof for hyperelliptic surfaces of types 2-7.

\begin{remark}\label{uw_idzie_na_innych_S}~

\textbullet \ For a hyperelliptic surface of even type there is no Case  IIb nor IIIb, as the divisor $(\mu/\gamma) B\equiv(0,1)$ is not effective on such a surfaces. Hence while proving that respectively  $M\widetilde{C}\geq 0$ and $N\widetilde{C}> 0$, we do not consider the curve $C\equiv (\mu/\gamma) B$.

\textbullet \ For a hyperelliptic surface of type  $3$, $4$, $7$ there are more types of singular fibres $mA/\mu$. Hence in Cases  I, IIa, IIIa, IV, and for a surfaces of type 3 and 7 also in Cases IIb and IIIb, we have to consider the intersection of $C$ with singular fibres of all admissible types $mA/\mu$, but they are estimated from below by the intersetion with a fibre $A/\mu$.

\textbullet \ For a hyperelliptic surface of type $3$, $4$ and $7$ we consider additional cases --- points $x_1$, $\ldots$, $x_s$ lie on a fixed singular fibre $mA/\mu$, where $1<m\leq\frac{\mu}2$. These cases are analogous to Cases IIIa, IIIb.
\end{remark}

Proof of Theorem \ref{k-dzet-szer} for $r=1$ gives us the following:
\begin{corollary}\label{gen-k-dzet}
Let $S$ be a hyperelliptic surface. Let $L$ be a line bundle of type  $(k+2,k+2)$ on $S$. Then $L$ generates $k$-jets at any point $x$.
\end{corollary}

\subsection*{Acknowledegments} The author would like to thank  Tomasz Szemberg and Halszka Tutaj-Gasi\'nska for many helpful discussions and improving the readability of the paper.

\end{document}